\documentclass{amsart}
\usepackage{amsmath}
\usepackage{amssymb}
\usepackage{amsthm}
\usepackage[msc-links,abbrev]{amsrefs}
\usepackage[noabbrev,capitalize]{cleveref}

\DeclareMathOperator{\tr}{tr}
\DeclareMathOperator{\End}{End}

\newcommand{\tfss}{\mathring{A}}
\newcommand{\oW}{\overline{W}}

\newcommand{\lp}{\langle}
\newcommand{\rp}{\rangle}
\newcommand{\lv}{\lvert}
\newcommand{\rv}{\rvert}

\newcommand{\bN}{\mathbb{N}}
\newcommand{\bR}{\mathbb{R}}

\def\sideremark#1{\ifvmode\leavevmode\fi\vadjust{\vbox to0pt{\vss
 \hbox to 0pt{\hskip\hsize\hskip1em
 \vbox{\hsize3cm\tiny\raggedright\pretolerance10000
 \noindent #1\hfill}\hss}\vbox to8pt{\vfil}\vss}}}

\newtheorem{theorem}{Theorem}[section]
\newtheorem{proposition}[theorem]{Proposition}
\newtheorem{lemma}[theorem]{Lemma}
\newtheorem{corollary}[theorem]{Corollary}

\numberwithin{equation}{section}

\begin{document}

% \title{A Rigidity Result for catenoids in spaces of constant curvature}
\title[A rigidity result for catenoids]{A sharp inequality for trace-free matrices with applications to hypersurfaces}
\author{Jeffrey S. Case}
\address{109 McAllister Building \\ Penn State University \\ University Park, PA 16802 \\ USA}
\email{jscase@psu.edu}
\author{Aaron J. Tyrrell}
\address{18A Department of Mathematics and Statistics \\ Texas Tech University \\ Lubbock \\ TX 79409 \\ USA}
\email{aatyrrel@ttu.edu}
% \date{\today}
\begin{abstract}
 We derive a sharp inequality relating the second and fourth elementary symmetric functions of the eigenvalues of a trace-free matrix and give two applications.
 First, we give a new proof of the classification of conformally flat hypersurfaces in spaceforms.
 Second, we construct a functional which characterizes rotational hypersurfaces and catenoids.
 % We make use of elementary symmetric functions in order to prove an inequality involving the second fundamental form of a minimal hypersurface in a space of constant curvature. We first prove that $p_3^2+4p_2^3 \leq 0$ and $p_4 + 3p_2^2 \geq 0$.
 % Moreover, the first holds with equality if and only if $A$ has an eigenspace of codimension at most one. This then allows us to show that a particular pointwise extrinsic curvature invariant is non-positive on minimal hypersurfaces and vanishes if and only if the hypersurface is a catenoid.
\end{abstract}
\maketitle

\section{Introduction}
\label{sec:intro}

Let $i \colon N^n \to (M^{n+1},g)$, $n \geq 4$, be an immersed hypersurface in a locally conformally flat Riemannian manifold.
This note is motivated by two interesting rigidity results for $N$ in terms of its principal curvatures.
First, Nishikawa and Maeda showed~\cite{NiskikawaMaeda1974} that $(N^n,i^\ast g)$ is locally conformally flat if and only if at each point it has a principal curvature of multiplicity at least $n-1$.
Second, do Carmo and Dajczer showed~\cite{do2012rotation} that, under the stronger assumption that $(M^{n+1},g)$ is complete, simply connected, and of constant sectional curvature $c \in \bR$, the hypersurface $N$ is contained in a so-called rotation hypersurface~\cite{do2012rotation}*{Definition~2.2} if at each point it has a principal curvature of multiplicity exactly $n-1$.
Furthermore, if $N$ is minimal, then it is contained in a catenoid if and only if at each point it has a principal curvature of multiplicity at least $n-1$.%\edz{I cut the other two references; put them back in if we can say something about them}

The purpose of this short note is to point out that the condition on the principal curvatures can be recast in terms of a sharp inequality relating the squared-norm $\lvert\tfss^2\rvert^2$ of the square of the trace-free part of the second fundamental form $\tfss$ and the square of the squared-norm $\lvert\tfss\rv^2$ of $\tfss$.
Our main result is a purely algebraic statement about trace-free linear maps on finite-dimensional vector spaces:

% Let $M^{n+1}(c)$ be a simply connected, complete Riemannian manifold of dimension $n$ with constant curvature $c,$ where $c$ is a real number. In 1983 Do Carmo and Dajczer \cite{do2012rotation} defined rotation hypersurfaces in these spaces. They computed their principal curvatures and gave a local characerization of such hypersurfaces in terms of them. Equation (3.13) in their paper is an ODE that is satisfied by rotation hypersurfaces with prescribed mean curvature in $M^n(c).$ Setting $c=0, n=2$ and setting the mean curvature to be $0,$ it reduces to the equation for a catenary. It is in this sense that minimal rotation hypersurfaces are generalizations of classical catenoids and for this reason we call them catenoids. A consequence of Corollary $4.4$ of their work is that a minimal hypersurface of $M^n(c)$ is a catenoid if and only if it has at least $n-1$ principal curvatures which are equivalent. Some other work on rotational hypersurfaces can be found here \cite{wang2018simons}, \cite{levitt1985symmetry}. The main result of this paper is the following:\edz{rewrite}

\begin{theorem}
    \label{main-thm}
    Let $V$ be an $n$-dimensional inner product space, $n \geq 4$, and let $\tfss \in \End(V)$ be trace-free.
    Then
    \begin{equation*}
        \lvert \tfss^2\rvert^2 \leq \frac{n^2-3n+3}{n(n-1)}\lvert \tfss\rvert^4
    \end{equation*}
    with equality if and only if $\tfss$ has an eigenspace of dimension at least $n-1$.
\end{theorem}

The main idea of the proof of \cref{main-thm} is that the Newton inequality relating the first, second, and third elementary symmetric functions of the eigenvalues of a matrix $A$ realizes equality if and only if $A$ is proportional to the identity or has rank $1$.
Applying this to $\tfss + \lambda I$ for $I$ the identity map yields the sharp inequality
\begin{equation*}
    \left(\tr\tfss^3\right)^2 \leq \frac{(n-2)^2}{n(n-1)}\lv\tfss\rv^6 
\end{equation*}
for any trace-free $\tfss \in \End(V^n)$;
this is sometimes called the $\lambda$-method~\cite{Mitrinovic1970}.
Repeating this argument for the Newton inequality relating the second, third, and fourth elementary symmetric functions yields \cref{main-thm}.

We conclude with a number of applications of \cref{main-thm} to the rigidity of locally conformally flat hypersurfaces, of rotation hypersurfaces, and of catenoids.

Our first application is a new proof of the characterization~\cite{NiskikawaMaeda1974} of locally conformally flat hypersurfaces in locally conformally flat manifolds in terms of their principal curvatures.

\begin{corollary}
    \label{lcf-characterization}
    Let $i \colon N^n \to (M^{n+1},g)$, $n \geq 4$, be an immersed hypersurface in a locally conformally flat manifold.
    The induced metric $i^\ast g$ on $N$ is locally conformally flat if and only if for each $p \in N$, the shape operator $A_p \colon T_pN \to T_pN$ has an eigenspace of dimension at least $n-1$.
\end{corollary}

Our second application is a rigidity result for catenoids in simply connected spaceforms amongst all minimial hypersurfaces.

\begin{corollary}
    \label{catenoid-characterization}
    Let $i \colon N^n \to M^{n+1}(c)$, $n \geq 4$, be a minimal hypersurface in the complete, simply connected Riemannian $(n+1)$-manifold of constant sectional curvature $c \in \bR$.
    Then the shape operator $A \colon TN \to TN$ satisfies
    \begin{equation}
        \label{eqn:catenoid-inequality}
        \lvert A^2 \rvert^2 \leq \frac{n^2-3n+3}{n(n-1)}\lvert A\rvert^4 .
    \end{equation}
    Moreover, equality holds if and only if $i(N)$ is contained in a catenoid.
\end{corollary}

The above results motivate the introduction of two energy functionals for hypersurfaces of a Riemannian manifold.
First, given an immersion $i \colon N^n \to (M^{n+1},g)$, we define the \emph{rotational energy} of $N$ by
\begin{equation*}
    E_{rot}[N] := \int_N \left( \frac{n^2-3n+3}{n(n-1)}\lvert \tfss\rvert^4 - \lvert \tfss^2 \rvert^2 \right) \, dA ,
\end{equation*}
assuming it is finite.
Our terminology is motivated by the characterization~\cite{do2012rotation} of rotation hypersurfaces in simply connected spaceforms.

\begin{corollary}
    \label{rotation-hypersurface-characterization}
    Let $i \colon N^n \to M^{n+1}(c)$, $n \geq 4$, be a nowhere umbilic immersed hypersurface in the complete, simply connected, Riemannian $(n+1)$-manifold of constant sectional curvature $c \in \bR$.
    Then $E_{rot}[N] \geq 0$ with equality if and only if $N$ is contained in a rotation hypersurface.
\end{corollary}

This result is sharper for minimal immersions.

\begin{corollary}
    \label{integral-catenoid-characterization}
    Let $i \colon N^n \to M^{n+1}(c)$, $n \geq 4$, be a minimal immersed hypersurface in the complete, simply connected, Riemannian $(n+1)$-manifold of constant sectional curvature $c \in \bR$.
    Then $E_{rot}[N] \geq 0$ with equality if and only if $N$ is contained in a catenoid.
\end{corollary}

Second, given an immersion $i \colon N^n \to (M^{n+1},g)$, we define the \emph{conformal rotational energy} of $N$ by
\begin{equation*}
    E_{rot}^{conf}[N] := \int_N \lv \tfss\rv^{n-4}\left( \frac{n^2-3n+3}{n(n-1)}\lvert \tfss\rvert^4 - \lvert \tfss^2 \rvert^2 \right) \, dA ,
\end{equation*}
assuming it is finite.
Note that $E_{rot}^{conf} = E_{rot}$ for four-dimensional hypersurfaces.
The main point is of this definition is that $E_{rot}^{conf}$ is conformally invariant, and hence gives a conformally invariant characterization of immersed submanifolds for which the shape operator has an eigenspace of dimension at least $n-1$.

\begin{corollary}
    \label{conformally-invariant-rotational-characterization}
    Let $i \colon N^n \to (M^{n+1},g)$, $n \geq 4$, be an immersed hypersurface in a Riemannian manifold.
    Then $E_{rot}^{conf}[N] \geq 0$ with equality if and only if for each $p \in N$, the shape operator $A_p \colon T_pN \to T_pN$ has an eigenspace of dimension at least $n-1$.
\end{corollary}

This note is organized as follows.
In \cref{sec:symmetric} we give an elementary proof of \cref{main-thm}.
In \cref{sec:applications} we prove our geometric applications of \cref{main-thm}.

\section{Elementary symmetric functions}
\label{sec:symmetric}

Let $A \in \End(V^n)$ be a linear map on an $n$-dimensional vector space.
Denote by
\begin{equation}
    \label{eqn:defn-sigmak}
    \sigma_k(A) := \sum_{i_1 < \dotsm < i_k} \lambda_{i_1} \dotsm \lambda_{i_k}
\end{equation}
the $k$-th elementary symmetric function of the eigenvalues $\{ \lambda_1, \dotsc, \lambda_n\}$ of $A$, with the convention $\sigma_0(A)=1$.
Note that $\sigma_1(A) = \tr A$.
Define $p_k(A)$, $k \leq n$, by
\begin{equation*}
 \sigma_k(A) = \binom{n}{k}p_k(A) ;
\end{equation*}
this renormalization is such that $p_k(\lambda I) = \lambda^k$, where $I$ is the identity map.
The main tool in the proof of \cref{main-thm} is the sharp version of Newton's inequalities:%\edz{I didn't cite Hardy et.\ al.\ because they mischaracterize the equality case.  But Rosset doesn't write it the way I would like, either}

\begin{lemma}[\cite{Rosset1989}*{Section~2}]
 \label{newton}
 Let $A \in \End(V^n)$ and let $k \leq n-1$ be a positive integer.
 Then
 \begin{equation*}
  p_k^2(A) \geq p_{k-1}(A) p_{k+1}(A)
 \end{equation*}
 with equality if and only if $A$ is proportional to the identity or $\dim\ker A \geq n-k+1$.
\end{lemma}

Applying the $\lambda$-method to \cref{newton} with $k=2$ and $k=3$ yields useful inequalities, the second of which is recorded in \cref{main-thm}.
We record the resulting inequalities separately due to their different dependence on $n = \dim V$.

Our first inequality relates $p_2(\tfss)$ and $p_3(\tfss)$ for a trace-free linear map $\tfss$ on a vector space of dimension at least $3$.

\begin{proposition}
 \label{first-prop}
 Let $\tfss \in \End(V^n)$, $n \geq 3$, be such that $p_1(\tfss)=0$.
 Then
 \begin{equation*}
    p_3^2(\tfss) + 4p_2^3(\tfss) \leq 0 .
 \end{equation*}
 Moreover, equality holds if and only if $\tfss$ has an eigenspace of dimension at least $n-1$.
\end{proposition}

\begin{proof}
 It is well-known, and follows easily from Equation~\eqref{eqn:defn-sigmak}, that
    \begin{equation}
         \label{eqn:recursive}
         p_k(A + \lambda I) = \sum_{j=0}^k \binom{k}{j}\lambda^jp_{k-j}(A)
    \end{equation}
    for all $A \in \End(V^n)$ and all $\lambda \in \bR$.

    Now let $\tfss \in \End(V^n)$ be such that $p_1(\tfss)=0$.
    \Cref{newton} implies that $p_2(\tfss) \leq 0$ with equality if and only if $\tfss = 0$.
    Since the conclusion is trivially true if $\tfss=0$, we may assume that $p_2(\tfss)<0$.
    For notational simplicity, denote $p_k := p_k(\tfss)$.

    Let $\lambda \in \bR$.
    Applying \cref{newton} to $\tfss+\lambda I$ and simplifying via Equation~\eqref{eqn:recursive} yields
    \begin{equation}
        \label{eqn:lambda-step1}
        \begin{split}
            0 & \leq ( p_2 + \lambda^2)^2 - \lambda(p_3 + 3\lambda p_2 + \lambda^3) \\
            & = p_2^2 - \lambda p_3 - \lambda^2p_2 ,
        \end{split}
    \end{equation}
    with equality if and only if $\tfss$ has an eigenspace of dimension at least $n-1$.
    In fact, since $p_1=0$ and $p_2<0$, the dimension must equal $n-1$ in the case of equality in Inequality~\eqref{eqn:lambda-step1}.
    Multiplying Inequality~\eqref{eqn:lambda-step1} by $4p_2 < 0$ yields
    \begin{align*}
        0 & \geq 4p_2^3 - 4\lambda p_2p_3 - 4\lambda^2p_2^2 \\
        & = -(2\lambda p_2 + p_3)^2 + 4p_2^3 + p_3^2 .
    \end{align*}
    Therefore $p_3^2 + 4p_2^3 \leq 0$ with equality if and only if $\tfss$ has an eigenspace of dimension $n-1$.
\end{proof}

Our second inequality relates $p_2(\tfss)$ and $p_4(\tfss)$ for a trace-free linear map $\tfss$ on a vector space of dimension at least $4$.

\begin{proposition}
 \label{main-prop}
 Let $\tfss \in \End(V^n)$, $n \geq 4$, be such that $p_1(\tfss)=0$.
 Then
 \begin{equation*}
    p_4(\tfss) + 3p_2^2(\tfss) \geq 0 .
 \end{equation*}
 Moreover, equality holds if and only if $\tfss$ has an eigenspace of dimension at least $n-1$.
\end{proposition}

\begin{proof}
    Let $\tfss \in \End(V^n)$ be such that $p_1(\tfss)=0$.
    For notational simplicity, denote $p_k := p_k(\tfss)$.
    \Cref{newton} implies that $p_2(\tfss)\leq0$ with equality if and only if $\tfss=0$.
    Since the conclusion is trivially true if $\tfss=0$, we may assume that $p_2<0$.

    Now, applying \cref{newton} and Equation~\eqref{eqn:recursive} to $\tfss+\lambda I$ yields
    \begin{align*}
        0 & \leq (p_3 + 3\lambda p_2 + \lambda^3)^2 - (p_2 + \lambda^2)(p_4 + 4\lambda p_3 + 6\lambda^2p_2 + \lambda^4) \\
        & = p_3^2 - p_2p_4 + 2\lambda p_2p_3 + \lambda^2(3p_2^2 - p_4) - 2\lambda^3 p_3 - \lambda^4 p_2
    \end{align*}
    for all $\lambda \in \bR$.
    Multiplying by $16p_2^3$ and denoting $s := 2\lambda p_2 + p_3$ yields
    \begin{align*}
        0 & \geq 16p_2^3p_3^2 - 16p_2^4p_4 + 32\lambda p_2^4p_3 + 16\lambda^2p_2^2(3p_2^3 - p_2p_4) - 32\lambda^3p_2^3p_3 - 16\lambda^4p_2^4 \\
        % & = 16p_2^3p_3^2 - 16p_2^4p_4 + 16sp_2^3p_3 - 16p_2^3p_3^2 + 4s^2(3p_2^3-p_2p_4) - 8sp_3(3p_2^3-p_2p_4) \\
        %    & \quad + 4p_3^2(3p_2^3-p_2p_4) - 4s^3p_3 + 12s^2p_3^2 - 12sp_3^3 + 4p_3^4 \\
        %    & \quad - s^4 + 4s^3p_3 - 6s^2p_3^2 + 4sp_3^3 - p_3^4 \\
        % & = -16p_2^4p_4 + 16sp_2^3p_3 + 12s^2p_2^3 - 4s^2p_2p_4 - 24sp_2^3p_3 + 8sp_2p_3p_4 \\
        %    & \quad + 12p_2^2p_3^2 - 4p_2p_3^2p_4 + 12s^2p_3^2 - 12sp_3^3 + 4p_3^4 - s^4 - 6s^2p_3^2 + 4sp_3^3 - p_3^3 \\
        & = (3p_3^2-4p_2p_4)(p_3^2 + 4p_2^3) + 8sp_3( p_2p_4- p_2^3 - p_3^2) + s^2(12p_2^3 - 4p_2p_4 + 6p_3^2) - s^4 .
    \end{align*}
    Setting $s=0$ yields
    \begin{equation}
        \label{eqn:lambda-step2}
        (3p_3^2 - 4p_2p_4)(p_3^2 + 4p_2^3) \leq 0 .
    \end{equation}
    \Cref{first-prop} implies that $p_3^2 + 4p_2^3 \leq 0$ with equality if and only if, up to a multiplicative constant and a choice of basis,
    \begin{equation*}
        \tfss = \begin{pmatrix}
            1 & 0 & \cdots & 0 & 0 \\
            0 & 1 & \cdots & 0 & 0 \\  
            \vdots & \vdots & \ddots & \vdots & \vdots \\
            0 & 0 & \cdots & 1 & 0 \\
            0 & 0 & \cdots & 0 & 1-n 
        \end{pmatrix} .
    \end{equation*}
    If $p_3^2 + 4p_2^3=0$, then a straightforward computation yields $p_k = 1-k$ for all $k \in \bN$, and hence $p_4 + 3p_2^2 = 0$.
    If instead $p_3^2 + 4p_2^3 < 0$, then Inequality~\eqref{eqn:lambda-step2} implies that $3p_3^2 - 4p_2p_4 \geq 0$.
    Therefore
    \begin{equation*}
        0 \leq 3p_3^2 - 4p_2p_4 < -12p_2^3 - 4p_2p_4 = -4p_2(p_4 + 3p_2^2) .
    \end{equation*}
    Since $p_2<0$, we conclude that $p_4 + 3p_2^2 > 0$.
\end{proof}

Expressing the nonnegative quantity $p_4 + 3p_2^2$ in terms of $\lv A\rv^4$ and $\lv A^2\rv^2$ yields \cref{main-thm}.

% \begin{corollary}
%     \label{fiberwise-corollary}
%     Let $\tfss \in \End(V^n)$, $n \geq 2$, be such that $p_1(\tfss)=0$.
%     Then
%     \begin{equation*}
%         \lv \tfss^2 \rv^2 \leq \frac{n^2-3n+3}{n(n-1)}\lv \tfss\rv^4 
%     \end{equation*}
%     with equality if and only if $\tfss$ has an eigenspace of dimension at least $n-1$.
% \end{corollary}

\begin{proof}[Proof of \cref{main-thm}]
    Let $\tfss \in \End(V^n)$ be such that $p_1(\tfss)=0$.
    Direct computation gives
    \begin{align*}
        \sigma_2(\tfss) & = -\frac{1}{2}\lv \tfss\rv^2 , \\
        \sigma_4(\tfss) & = \frac{1}{8}\lv \tfss\rv^4 - \frac{1}{4}\lv \tfss^2\rv^2 .
    \end{align*}
    Therefore
    \begin{equation*}
        0 \leq \binom{n}{4}\left( p_4(\tfss) + 3p_2^2(\tfss) \right) = -\frac{1}{4}\left( \lv \tfss^2\rv^2 - \frac{n^2-3n+3}{n(n-1)}\lv \tfss\rv^4 \right) .
    \end{equation*}
    The conclusion readily follows from \cref{main-prop}.
\end{proof}

\section{Geometric applications}
\label{sec:applications}

We conclude by proving our geometric applications of \cref{main-thm}.

\begin{proof}[Proof of \cref{lcf-characterization}]
    Let $i \colon N^n \to (M^{n+1},g)$ be a Riemannian immersion into a locally conformally flat manifold.
    Denote by $\oW$ the Weyl tensor of $i^\ast g$.
    Denote by
    \begin{equation*}
        F := \frac{1}{n-2}\left( \tfss^2 - Gi^\ast g \right) 
    \end{equation*}
    the Fialkow tensor $F \in \Gamma(S^2T^\ast N)$, where $G := \tr_{i^\ast g} F = \frac{1}{2(n-1)}\lv\tfss\rv^2$ is its trace (cf.\ \cite{Fialkow1944}*{Equation~(13.9)}).
    The Gauss--Codazzi equations (cf.\ \cite{Fialkow1944}*{Equation~(22.13)}) imply that
    \begin{equation*}
        \oW = \frac{1}{2}\tfss \wedge \tfss + F \wedge g ,
    \end{equation*}
    where $S \wedge T$ denotes the Kulkarni--Nomizu product
    \begin{equation*}
        (S \wedge T)_{abcd} := S_{ac}T_{bd} + S_{bd}T_{ac} - S_{ad}T_{bc} - S_{bc}T_{ad} .
    \end{equation*}
    Direct computation gives (cf.\ \cite{Labbi2005})
    \begin{align*}
        \lv \tfss \wedge \tfss \rv^2 & = 8\lv\tfss\rv^4 - 8\lv\tfss^2\rv^2 , \\
        \lp \tfss \wedge \tfss, F \wedge g \rp & = -8\lp \tfss^2, F \rp , \\
        \lv F \wedge g \rv^2 & = 4\lp \tfss^2 , F \rp , \\
        \lp \tfss^2, F \rp & = \frac{1}{n-2}\lv\tfss^2\rv^2 - \frac{1}{2(n-1)(n-2)}\lv\tfss\rv^4 .
    \end{align*}
    Therefore
    \begin{equation*}
        \lv\oW\rv^2 = \frac{2(n^2-3n+3)}{(n-1)(n-2)}\lv\tfss\rv^4 -\frac{2n}{n-2}\lv\tfss^2\rv^2 .
    \end{equation*}
    The conclusion readily follows from \cref{main-thm}.
\end{proof}

\begin{proof}[Proof of \cref{catenoid-characterization}]
    Let $i \colon N^n \to M^{n+1}(c)$ be a minimal hypersurface.
    By minimality, the shape operator is trace-free;
    i.e.\ $A = \tfss$.
    \Cref{main-thm} implies that $A$ satisfies Inequality~\eqref{eqn:catenoid-inequality} with equality if and only if $A$ has an eigenvalue of multiplicity at least $n-1$.
    The final conclusion follows from a result~\cite{do2012rotation}*{Corollary~4.4} of do Carmo and Dajczer.
\end{proof}

\begin{proof}[Proof of \cref{rotation-hypersurface-characterization}]
    Let $i \colon N^n \to M^{n+1}(c)$ be a nowhere umbilic immersed hypersurface.
    \Cref{main-thm} implies that $E_{rot}[N] \geq 0$ with equality if and only if at each point $p \in M$ the shape operator $A_p$ has an eigenspace of dimension exactly $n-1$.
    The final conclusion follows from a result~\cite{do2012rotation}*{Theorem~4.2} of do Carmo and Dajczer.
\end{proof}

\begin{proof}[Proof of \cref{integral-catenoid-characterization}]
    Let $i \colon N^n \to M^{n+1}(c)$ be a minimal immersed hypersurface.
    \Cref{catenoid-characterization} implies that $E_{rot}[N] \geq 0$ with equality if and only if $i(N)$ is contained in a catenoid.
\end{proof}

\begin{proof}[Proof of \cref{conformally-invariant-rotational-characterization}]
    Let $i \colon N^n \to (M^{n+1},g)$ be an immersed hypersurface.
    \Cref{main-thm} implies that $E_{rot}^{conf}[N] \geq 0$ with equality if and only if
    \begin{equation*}
        \lvert \tfss\rvert^{n-4}\left( \frac{n^2-3n+3}{n(n-1)}\lvert\tfss\rvert^4 - \lvert\tfss^2\rvert^2 \right) = 0 .
    \end{equation*}
    It follows that for each $p \in N$, either $\tfss_p = 0$ or
    \begin{equation*}
        \lvert \tfss_p^2 \rvert^2 = \frac{n^2-3n+3}{n(n-1)}\lvert\tfss_p\rvert^4 .
    \end{equation*}
    In the former case, $\tfss_p$ has an eigenspace of dimension $n$.
    In the latter case, \cref{main-thm} implies that $\tfss_p$ has an eigenspace of dimension at least $n-1$.
\end{proof}

\section*{Acknowledgements}
We thank Eudes Leite de Lima for pointing out a mistake in an early version of this paper.
JSC was partially supported by the Simons Foundation (Grant \#524601).

\bibliographystyle{amsplain}
\bibliography{bibliography.bib}

\end{document}